\documentclass{article}

\usepackage{amssymb, amsmath, mathrsfs}
\usepackage{color}
\usepackage{pgf,tikz}
\usepackage{graphicx}
\usepackage{hyperref}

\makeatletter
\newcommand{\labeltarget}[1]{\Hy@raisedlink{\hypertarget{#1}{}}}
\makeatother

\newtheorem{theorem}{Theorem}[section]
\newtheorem{lemma}[theorem]{Lemma}
\newtheorem{proposition}[theorem]{Proposition}

\newtheorem{remark}[theorem]{Remark}

\newtheorem{proof}[theorem]{Proof}
\numberwithin{equation}{section}

\usepackage[]{graphicx}

		\title{Existence of sign-changing solution for a problem involving the fractional Laplacian with critical growth nonlinearities}
		
		   \date{}

		   \author{
		   	R.F. Gabert\thanks{R.F. Gabert was supported 
		   		in part by  CAPES. E-mail: rodrigogabert@dm.ufscar.br}; \,
		   	{R.S. Rodrigues\thanks{E-mail: rodrigo@dm.ufscar.br}}}

\begin{document}
		\pretolerance10000 

	\maketitle

	\begin{abstract}
		
	We study the existence of least energy sign-changing solution for the fractional equation $ (-\Delta)^{s} u= |u|^{2_{s}^{*}-2}u+\lambda f(x,u) $ in a smooth bounded domain $ \Omega $ of $ \mathbb{R}^{N} $, $ u=0 $ in $ \mathbb{R}^{N}\setminus \Omega $, where $ s\in (0,1) $ and $ 2_{s}^{*} $ is the fractional critical Sobolev exponent. The proof is based on constrained minimization in a subset of Nehari manifold, containing all
	the possible sign-changing solutions of the equation.
		
	\end{abstract}

	{\scriptsize \noindent{\bf Keywords}
	{Fractional Laplacian, sign-changing solutions, critical exponents.}}

{\scriptsize \noindent{\bf 2010 Mathematical Subject Classifications:}
	49J35, 35S15, and 35B33.}

\section{Introduction}\label{Section1}

In this paper, we establish the existence of a least energy sign-changing (nodal) solution of following problem:

\begin{equation}\label{Equa1}
\left \{
\begin{array}{ll}
(-\Delta)^s u=|u|^{2^{*}_{s}-2}u+\lambda f(x,u) & \operatorname{in}\, \Omega, \\
u=0  &\operatorname{in}\,  \mathbb{R}^{N} \setminus \Omega,\\
\end{array}
\right.
\end{equation}
where $ \Omega \subset \mathbb{R}^{N} $ is a smooth bounded domain, $ s\in (0,1) $, $ 2s<N $, $ \lambda >0 $ and $ 2_{s}^{*}=\frac{2N}{N-2s} $ is the fractional critical Sobolev exponent.

Let  $ \mathscr{S}(\mathbb{R}^{N}) $ be the Schwartz space.
If $ u\in \mathscr{S}(\mathbb{R}^{N}) $, we define the fractional Laplacian by
\begin{equation*}
(-\Delta)^{s}u(x):=c(N,s)\lim_{\epsilon \rightarrow 0}\int_{\mathbb{R}^{N}\setminus B(x,\epsilon)}\frac{u(x)-u(y)}{|x-y|^{N+2s}}dx dy,\, x\in \mathbb{R}^{N},
\end{equation*} 
where $ C(N,s) $ is a positive normalizing constant.

We assume that $ f:\overline{\Omega}\times \mathbb{R} \rightarrow \mathbb{R} $ is continuous so that $ f(x,\cdotp): \mathbb{R} \rightarrow \mathbb{R} $ is continuously differentiable, for a.e. $ x\in \Omega $ and satisfies the following assumptions:

\begin{itemize}
\item[(H1)] \labeltarget{1} There exist $ q\in (2,2_{s}^{*}) $ and $ c_{1}>0 $ such that 
\begin{equation*}
|\partial_{t}f(x,t)|\leq c_{1}(1+|t|^{q-2}),
\end{equation*}
for a.e. $ x\in \Omega $ and all $ t\in \mathbb{R} $.
\item[(H2)] \labeltarget{2} $ \lim\limits_{t\rightarrow 0}\frac{f(x,t)}{t}=0 $, uniformly for a.e. $ x\in \Omega $.
\item[(H3)] \labeltarget{3} There exist $ \mu \in(2,2_{s}^{*})$ such that 
\begin{equation*}
0<\mu F(x,t)\leq t f(x,t), 
\end{equation*}
for a.e. $ x\in \Omega $ and all $  t\neq 0 $, where $ F(x,t):=\int_{0}^{t}f(x,\tau) d\tau  $. 
\item[(H4)] \labeltarget{4} $ \frac{f(x,t)}{t} $ is increasing in $ t>0 $ and decreasing in $ t<0 $, for a.e. $ x \in \Omega $.
\end{itemize}

\begin{remark}\label{Obs1}
Note that from assumptions \hyperlink{3}{(H3)} and \hyperlink{4}{(H4)}, we can conclude that the function $ H(x,t):= t f(x,t)-2 F(x,t) $ is non-negative, increasing in $ t>0 $ and decreasing in $ t<0 $. Thus,
\begin{equation*}
t^{2}\partial _{t} f(x,t)-t f(x,t)= \partial _{t}H(x,t)t >0,
\end{equation*} 
for a.e. $  x\in \Omega $ and all $ t\neq 0 $.
\end{remark}

Recently, in literature, many works have been published involving nonlocal operators such as the fractional Laplacian operator, which get existence, non-existence and regularity results and, also, obtain qualitative properties of the solutions. In \cite{H}, a pioneer work, Caffarelli and Silvestre have given an approach local to the fractional Laplacian operator across $ s $-harmonic extension technique. This paper motived other authors to produce works involving this operator, establishing similar results to classic results obtained for the Laplacian operator. See, for instance, \cite{O}, \cite{P}, \cite{Q}, \cite{R}, and \cite{S}. On the other hand, Servadei and Valdinoci, in \cite{T} and \cite{U}, work with the fractional Laplacian as singular integral operator, where they presented a suitable fractional Sobolev space and a variational formulation. From this approach, a great amount of papers was developed involving fractional problems, see for example \cite{AB}, \cite{AA}, \cite{AC}, and references therein, among others. In particular, we use this last approach in our article.

The fractional Laplacian operator appears in diverse areas such as mathematical finances, quantum mechanics, water waves, phase transition, minimal surface, population dynamics, optimal control, game theory, L\'{e}vy processes in probability theory, among others. For more details about these subjects and them applications see \cite{A}, \cite{B}, \cite{L}, \cite{M}, \cite{N} and the references therein.

We will present hereinafter the fractional Sobolev space and the weak formulation of problem \eqref{Equa1}. For any measurable function $ u:\mathbb{R}^{N}\rightarrow \mathbb{R} $, we define the Gagliardo seminorm by setting
\begin{equation*}
[u]_{s,2}:= \left(\int_{\mathbb{R}^{N}}\frac{|u(x)-u(y)|^{2}}{|x-y|^{N+2s}}dx dy \right)^{\frac{1}{2}}
\end{equation*}
and, the fractional Sobolev space, we define by 
\begin{equation*}
W^{s,2}(\mathbb{R}^{N}):=\{u \in L^{2}(\mathbb{R}^{N}); [u]_{s,2}<\infty\},
\end{equation*}
provided with the norm
\begin{equation*}
\|u\|_{W^{s,2}(\mathbb{R}^{N})}:=\left(\|u\|_{L^{2}(\mathbb{R}^{N})}^{2}+[u]_{s,2}^{2}\right)^{\frac{1}{2}},
\end{equation*}
where $ \|\cdotp\|_{L^{2}(\mathbb{R}^{N})} $ is the norm in $ L^{2}(\mathbb{R}^{N}) $ (see \cite{D} for more details).

This space, with the inner product
\begin{equation*}
\langle u,v\rangle_{W^{s,2}}:=\langle u,v\rangle_{L^{2}(\mathbb{R}^{N})}+\int_{\mathbb{R}^{2N}}\frac{(u(x)-u(y))(v(x)-v(y))}{|x-y|^{N+2s}}dxdy,
\end{equation*}
is a Hilbert space, where
\begin{equation*}
\langle u,v\rangle_{L^{2}(\mathbb{R}^{N})}:=\int_{\mathbb{R}^{N}}uv dx .
\end{equation*}

In the course of text, we use the notation 
\begin{equation*}
W^{s,2}(\mathbb{R}^{N}):=\mathbb{H}^{s}(\mathbb{R}^{N}).
\end{equation*}

Since our problem involves a bounded domain $ \Omega \subset \mathbb{R}^{N} $, we will introduce the space

\begin{equation*}
\mathbb{H}_{0}^{s}(\Omega):=\{u\in \mathbb{H}^{s}(\mathbb{R}^{N});\, u(x)=0, \,\operatorname{a.e.}\, x\in \Omega \setminus\mathbb{R}^{N} \},
\end{equation*}
which is provided with the inner product 
\begin{equation*}
\langle u,v\rangle_{\mathbb{H}_{0}^{s}(\Omega)}:=\int_{\mathbb{R}^{2N}}\frac{(u(x)-u(y))(v(x)-v(y))}{|x-y|^{N+2s}}dx dy .
\end{equation*}

In the space $ \mathbb{H}_{0}^{s}(\Omega) $, this inner product produces a norm $ \|\cdot\|_{\mathbb{H}_{0}^{s}(\Omega)}=[\,\cdot\,]_{s,2} $, which is equivalent to norm $ \|\cdot\|_{\mathbb{H}^{s}(\mathbb{R}^{N})} $. By \cite[Lemma 1.31]{F}, the embedding $ \mathbb{H}_{0}^{s}(\Omega) \hookrightarrow L^{r}(\Omega)$ is continuous, for all $ r\in [1,2_{s}^{*}] $ and compact for all $ r\in [1,2_{s}^{*}) $. We denote by
\begin{equation}\label{EquabestC}
S_{s}:=\inf _{u\in \mathbb{H}_{0}^{s}(\Omega)\setminus \{0\}}\frac{\|u\|_{\mathbb{H}_{0}^{s}(\Omega)}^{2}}{\|u\|_{L^{2_{s}^{*}}(\Omega)}^{2}}
\end{equation}
the best constant corresponding to the fractional Sobolev embedding when $ r=2_{s}^{*} $.

We say that a function $ u\in \mathbb{H}_{0}^{s}(\Omega) $ is a weak solution of the problem \eqref{Equa1}, if
\begin{equation*}
\int_{\mathbb{R}^{2N}}\frac{(u(x)-u(y))(v(x)-v(y))}{|x-y|^{N+2s}}dx dy=\int_{\Omega}|u|^{2_{s}^{*}-2}u v dx+\lambda \int_{\Omega}f(x,u)v dx ,
\end{equation*}
for all $ v \in \mathbb{H}_{0}^{s}(\Omega) $.

From this weak formulation, we  infer that the weak solutions of \eqref{Equa1} are precisely the critical points of the energy functional
\begin{equation*}
I_{\lambda}(u):= \frac{1}{2}\|u\|_{\mathbb{H}_{0}^{s}(\Omega)}^{2}- \frac{1}{2_{s}^{*}}\int _{\Omega}|u|^{2_{s}^{*}}dx- \lambda \int _{\Omega}F(x,u)dx,
\end{equation*}
whose derivative is given by
\begin{equation*}
\langle I_{\lambda}'(u),v\rangle=\int _{\mathbb{R}^{2N}}\frac{(u(x)-u(y))(v(x)-v(y))}{|x-y|^{N+2s}}dx dy- \int _{\Omega} |u|^{2_{s}^{*}-2}uv dx-\lambda \int _{\Omega}f(x,u)v dx .
\end{equation*}

We say that $ u\in \mathbb{H}_{0}^{s}(\Omega) $ is a least energy sign-changing solution of problem \eqref{Equa1} if $ u $ is a weak solution of \eqref{Equa1} with $ u^{+}\neq 0 $, $ u^{-}\neq 0 $ and  
\begin{equation*}
I_{\lambda}(u)=\inf \{I_{\lambda}(v);I'_{\lambda}(v)=0, v^{+}\neq 0, v^{-}\neq 0\},
\end{equation*}
where $ u^{+}(x)=\max \{u(x),0\} $ and $ u^{-}(x)=\min \{u(x),0\} $.

Our main result is the following:
\begin{theorem}\label{Theo1}
Suppose \hyperlink{1}{(H1)}-\hyperlink{4}{(H4)} hold. Then, there exist $ \lambda^{*}>0 $ such that for all $ \lambda \geq\lambda^{*} $, problem \eqref{Equa1} has a least energy sign-changing solution. 
\end{theorem}

A standard method to find sign-changing solution is the application of minimax arguments in invariant sets of the descending flow (see \cite{Y}, \cite{Z}, \cite{K}), as well as the Lusternik-Schnirelman method (see \cite{X}, \cite{W}). However, these techniques require that the energy functional $ J $ in question decomposes $ u^{+} $ and $ u^{-} $ as follows:
\begin{equation*}
J(u)=J(u^{+})+J(u^{-}),
\end{equation*}
\begin{equation*}
\langle J'(u),u^{+}\rangle=\langle J'(u^{+}),u^{+}\rangle\; \operatorname{and}\; \langle J'(u),u^{-}\rangle=\langle J'(u^{-}),u^{-}\rangle.
\end{equation*}
In $ \mathbb{H}_{0}^{s}(\Omega) $, the functions $ u^{+}$ and $ u^{-}$ are not orthogonal, when both are non-trivial, because
\begin{equation*}
\langle u^{+},u^{-}\rangle _{\mathbb{H}_{0}^{s}(\Omega)}= -2\int _{\mathbb{R}^{2N}}\frac{u^{+}(x)u^{-}(y)}{|x-y|^{N+2s}}dx dy>0.
\end{equation*}
Hence, the energy functional $ I_{\lambda} $ does not satisfy this decomposition, verifying only the inequalities
\begin{equation*}
I_{\lambda}(u)>I_{\lambda}(u^{+})+I_{\lambda}(u^{-}),
\end{equation*}
\begin{equation*}
\langle I'_{\lambda}(u),u^{+}\rangle>\langle I'_{\lambda}(u^{+}),u^{+}\rangle\; \operatorname{and}\; \langle I'_{\lambda}(u),u^{-}\rangle>\langle I'_{\lambda}(u^{-}),u^{-}\rangle.
\end{equation*}
Thus, the techniques become unsuitable for the problem \eqref{Equa1}. Motived by \cite{E}, our strategy is to find a minimizer for $ I_{\lambda} $ in the set
\begin{equation*}
\mathcal{M}_{\lambda}:=\{u\in \mathbb{H}_{0}^{s}(\Omega); \langle I_{\lambda}'(u),u^{+}\rangle=\langle I_{\lambda}'(u),u^{-}\rangle=0, u^{+}\neq 0, u^{-}\neq 0\},
\end{equation*}
which contains all the possible sign-changing solutions of \eqref{Equa1} and is a subset of Nehari Manifold associated with functional $ I_{\lambda},$ namely
\begin{equation*}
\mathcal{N}_{\lambda}:=\{u\in \mathbb{H}_{0}^{s}(\Omega); \langle I_{\lambda}'(u),u\rangle=0, u\neq 0\}.
\end{equation*}

In \cite{V}, the authors also used this strategy to find a sign-changing solution for a problem involving the fractional Laplacian, but with non-critical nonlinearities. In our work, since appears a term with critical exponent in the nonlinearity, there is a difficulty to proof that the infimum 
\begin{equation*}
\gamma_{\lambda}:=\inf _{u\in \mathcal{M}_{\lambda}} I_{\lambda}(u)
\end{equation*}
is achieved in $ \mathcal{M}_{\lambda} $, because is not immediate that the minimizing sequence has convergent subsequence in $ L^{2_{s}^{*}}(\Omega) $, once the embedding $ \mathbb{H}_{0}^{s}(\Omega)\hookrightarrow L^{2_{s}^{*}}(\Omega) $, in general, is not compact, and this is essential in the proof. So, to obtain this convergence, we need to lower the level $ \gamma _{\lambda} $ for a level beneath of $\beta := \frac{s}{N}S_{s}^{\frac{N}{2s}} $, which is possible when $ \lambda $ is big enough.

In Section \ref{Section2}, we will present some results which are consequences of ours assumptions and others which are essential in proof of main result. Proof of the main result will be given in Section \ref{Section3}.

\section{Auxiliaries results}\label{Section2}

\begin{lemma}\label{Lem1}
Suppose that $ f: \overline{\Omega}\times \mathbb{R} \rightarrow \mathbb{R} $ satisfies the conditions \hyperlink{1}{(H1)}-\hyperlink{2}{(H2)}. Then, for any $ \epsilon >0 $, there exist $ \delta(\epsilon)>0 $ such that
\begin{equation*}
|f(x,t)|\leq \epsilon |t|+ \delta (\epsilon) |t|^{q-1},
\end{equation*}
for a.e. $  x\in \Omega $ and all $ t\in \mathbb{R} $.
\end{lemma}

\begin{proof}
Fixed $ \epsilon>0 $, by \hyperlink{2}{(H2)}, there exist $ \sigma(\epsilon) >0 $ such that if $ |t|< \sigma$, then
\begin{equation}\label{Equa2}
|f(x,t)|<\epsilon |t|,
\end{equation}
for a.e. $  x\in \Omega $. On the other hand, by \hyperlink{1}{(H1)}, if $ |t|\geq \sigma $, then 
\begin{eqnarray}\label{Equa3}
|\partial _{t}f(x,t)|&\leq &c_{1} (1+|t|^{q-2})\nonumber\\
&\leq& c_{1}(\sigma ^{2-q}|t|^{q-2}+|t|^{q-2})\nonumber\\
&\leq&\delta'(\sigma)|t|^{q-2},
\end{eqnarray}
for a.e. $  x\in \Omega $ and all $ t\in \mathbb{R} $, where $ \delta '(\sigma)=c_{1}(\sigma^{2-q}+1)>0 $.
By integrating from $ 0 $ to $ t $ in  \eqref{Equa3}, we get
\begin{equation}\label{Equa4}
|f(x,t)|\leq \delta(\sigma)|t|^{q-1},
\end{equation}
for a.e. $  x\in \Omega $ and all $ t\in \mathbb{R} $.
Therefore, from \eqref{Equa2} and \eqref{Equa4}, it follows that
\begin{equation*}
|f(x,t)|\leq \epsilon |t|+\delta(\epsilon)|t|^{q-1},
\end{equation*} 
for a.e. $  x\in \Omega $ and all $ t\in \mathbb{R} $.

\end{proof}

\begin{lemma}\label{Lem2}
Suppose that $ f:\overline{\Omega} \times \mathbb{R} \rightarrow \mathbb{R} $ satisfies \hyperlink{3}{(H3)}. Then, there are $ c_{1}, \, c_{2}>0 $ such that
\begin{equation*}
F(x,t)\geq c_{1}|t|^{\mu}- c_{2},
\end{equation*}
for a.e. $  x\in \Omega $ and all $ t\in \mathbb{R} $.
\end{lemma}

\begin{proof}
By assumption \hyperlink{3}{(H3)}, for some $ r>0 $, we get
\begin{equation}\label{R1}
\frac{tf(x,t)}{F(x,t)}\geq \mu,
\end{equation} 
for a.e. $ x\in \Omega $ and all $ |t|\geq r $.

Firstly, suppose $ t >r $. By integrating (\ref{R1}) in $ [r,t] $, we deduce that
\begin{equation*}
F(x,t)\geq \frac{F(x,r)}{r^{\mu}} t^{\mu},
\end{equation*}
for a.e. $ x\in \Omega $ and all $ t>r $.
Since the function $ x\mapsto F(x,t) $ is continuous on $ \overline{\Omega} $, we can infer that
\begin{equation*}
F(x,t)\geq \min_{x\in \overline{\Omega}} F(x,r)r^{-\mu} t^{\mu},
\end{equation*}
for a.e. $  x\in \overline{\Omega} $ and all $ t>r $.
On the other hand, if $ t<-r $, by integrating in $ [t,-r] $, we get 
\begin{equation*}
F(x,t)\geq \frac{F(x,-r)}{r^{\mu}}|t|^{\mu}\geq \min_{x\in \overline{\Omega}} F(x,-r)r^{-\mu}|t|^{\mu}, 
\end{equation*}
for a.e. $ x\in \overline{\Omega} $ and all $ t<-r $.

Therefore, there exist $ c_{1}>0 $ such that
\begin{equation*}
F(x,t)\geq c_{1}|t|^{\mu},
\end{equation*}
for a.e. $  x\in \overline{\Omega} $ and all $ t $ such that $ |t|\geq r $.

Since $ F $ is continuous in $ \overline{\Omega}\times \mathbb{R} $, there exist $ \tilde{c}_{2}>0 $ such that
\begin{equation*}
|F(x,t)|\leq \tilde{c}_{2},
\end{equation*}
for all $ (x,t) \in \overline{\Omega}\times \{t\in \mathbb{R}; |t|\leq r\} $.

Take $ c_{2}=\tilde{c}_{2}+c_{1}r^{\mu} >0$. Then,
\begin{equation*}
F(x,t)\geq c_{1}|t|^{\mu}-c_{2},
\end{equation*}
for a.e. $ x\in \Omega $ and all $ t \in \mathbb{R} $.
\end{proof}

\begin{lemma}\label{lem3}
The set $ \mathcal{M}_{\lambda} $ is non-empty. More precisely, given $ u\in \mathbb{H}_{0}^{s}(\Omega) $, with $ u^{+}\neq 0$ and $ u^{-}\neq 0 $, there are $ t_{u}>0$ and $\theta_{u}>0 $ such that 
\begin{equation*}
\langle I'(t_{u}u^{+}+\theta_{u}u^{-}),u^{+}\rangle=0=\langle I'(t_{u}u^{+}+\theta_{u}u^{-}),u^{-}\rangle.
\end{equation*}
\end{lemma}

\begin{proof}
For any $ u\in \mathbb{H}_{0}^{s}(\Omega) $, with $ u^{+}\neq 0$ and $ u^{-}\neq 0 $, we consider $ \Phi_{u}:[0,\infty)\times[0,\infty)\rightarrow \mathbb{R}^2$, given by
\begin{equation}\label{Equa5}
\Phi_{u}(t,\theta):=(\langle I'(t u^{+}+\theta u^{-}),t u^{+}\rangle,\langle I'(t u^{+}+\theta u^{-}),\theta u^{-}\rangle). 
\end{equation}

Since $ I_{\lambda}\in C^{1}(\mathbb{H}_{0}^{s}(\Omega),\mathbb{R}) $, follows that $ \Phi_{u} \in C([0,\infty)\times [0,\infty),\mathbb{R}^{2})$. We will show that there are $ \delta_{1}>0$ and $ \delta_{2}>0 $ such that the equation $ \Phi_{u}(t,\theta)=(0,0) $ has solution in the square $ R=[\delta_{1},\delta_{2}]\times [\delta_{1},\delta_{2}] $. For this conclusion, we will use the Miranda's Theorem in $ \mathbb{R}^{2} $ (see \cite{G}). To simplify the notation, we will write $ \Phi_{u}= (\Phi_{u}^{1},\Phi_{u}^{2}) $. 

We need to check the hypotheses of Miranda's Theorem, namely
\begin{eqnarray*}
\Phi_{u}^{1}(\delta_{1},\theta)>0 \mbox{ and }
\Phi_{u}^{1}(\delta_{2},\theta)<0,\,\forall \theta \in [\delta_{1},\delta_{2}],\\
\Phi_{u}^{2}(t,\delta_{1}) >0
\mbox{ and }
\Phi_{u}^{2}(t,\delta_{2})<0,\,\forall t\in [\delta_{1},\delta_{2}].
\end{eqnarray*}

Now, from Lemma \ref{Lem1}, using that the embedding $ \mathbb{H}_{0}^{s}(\Omega)\hookrightarrow L^{r}(\Omega) $ is continuous, for all $ r\in [1,2_{s}^{*}] $, we obtain
\begin{eqnarray*}
\langle I'_{\lambda}(tu^{+}+\theta u^{-}),t u^{+}\rangle &=& \langle tu^{+}+\theta u^{-},t u^{+}\rangle_{\mathbb{H}_{0}^{s}(\Omega)}-t^{2_{s}^{*}}\int _{\Omega}|u^{+}|^{2_{s}^{*}} dx-\lambda \int _{\Omega} f(x,tu^{+})t u^{+}dx\\
&=& t^{2}\|u^{+}\|_{\mathbb{H}_{0}^{s}(\Omega)}^{2}-2t\theta \int _{\mathbb{R}^{2N}}\frac{u^{+}(x)u^{-}(y)}{|x-y|^{N+2s}}dx dy-t^{2_{s}^{*}}\int_{\Omega} |u^{+}|^{2_{s}^{*}}dx\\
&&-\lambda \int _{\Omega}f(x,tu^{+})tu^{+} dx\\
&\geq & t^{2}\|u^{+}\|_{\mathbb{H}_{0}^{s}(\Omega)}^{2}- t^{2_{s}^{*}}\int _{\Omega}|u^{+}|^{2_{s}^{*}}dx -\lambda \epsilon t^{2}\int _{\Omega}|u^{+}|^{2}dx-\lambda \delta (\epsilon)t^{q}\int _{\Omega}(u^{+})^{q}dx\\
&\geq & t^{2}\|u^{+}\|_{\mathbb{H}_{0}^{s}(\Omega)}- \tilde{c}_{1}t^{2_{s}^{*}}\|u^{+}\|_{\mathbb{H}_{0}^{s}(\Omega)}^{2_{s}^{*}}-\lambda \epsilon \tilde{c}_{1}t^{2}\|u^{+}\|_{\mathbb{H}_{0}^{s}(\Omega)}^{2}-\lambda \delta (\epsilon)\tilde{c}_{1}t^{q}\|u^{+}\|_{\mathbb{H}_{0}^{s}(\Omega)}^{q}\\
&\geq& (1-\lambda \tilde{c}_{1}\epsilon)t^{2}\|u^{+}\|_{\mathbb{H}_{0}^{s}(\Omega)}^{2}-\tilde{c}_{1}t^{2_{s}^{*}}\|u^{+}\|_{\mathbb{H}_{0}^{s}(\Omega)}^{2_{s}^{*}}-\lambda \delta(\epsilon)\tilde{c}_{1}t^{q}\|u^{+}\|_{\mathbb{H}_{0}^{s}(\Omega)}^{q},
\end{eqnarray*}
for some $ \tilde{c}_{1}>0 $. We choose $ \epsilon>0 $ such that $ (1-\lambda \tilde{c}_{1}\epsilon)>0 $. Since $ 2_{s}^{*}>2$ and $q>2 $, it follows that for any $ \theta\geq 0 $, $ \langle I'(tu^{+}+\theta u^{-}),tu^{+}\rangle>0 $, if $ t\ll 1 $.

Analogously, we can prove that, for any $ t \geq 0 $,  $ \langle I'_{\lambda}(tu^{+}+\theta u^{-}),\theta u^{-}\rangle> 0 $, if $ \theta \ll 1 $.

Therefore, there exist $ \delta_{1}>0 $ such that 
\begin{equation}\label{Equa6}
\Phi_{u}^{1}(\delta_{1},\theta)>0 \mbox{ and } \Phi_{u}^{2}(t,\delta_{1})>0, 
\end{equation}
for all $ t,\theta \geq 0 $.

From Lemma \ref{Lem2} and assumption \hyperlink{3}{(H3)}, it follows that 
\begin{eqnarray*}
\langle I'_{\lambda}(tu^{+}+\theta u^{-}),tu^{+}\rangle&=&t^{2}\|u^{+}\|_{\mathbb{H}_{0}^{s}(\Omega)}-2t\theta \int _{\mathbb{R}^{2N}}\frac{u^{+}(x)u^{-}(y)}{|x-y|^{N+2s}}dx dy - t^{2_{s}^{*}}\int _{\Omega}|u^{+}|^{2_{s}^{*}}dx\\
&&-\lambda \int _{\Omega}f(x,tu^{+})tu^{+}dx\\
&\leq& t^{2}\|u^{+}\|_{\mathbb{H}_{0}^{s}(\Omega)}-2t\theta \int _{\mathbb{R}^{2N}}\frac{u^{+}(x)u^{-}(y)}{|x-y|^{N+2s}}dx dy- t^{2_{s}^{*}}\int _{\Omega}|u^{+}|^{2_{s}^{*}}dx\\
&&- t^{\mu}\lambda c_{1}\int _{\Omega}|u^{+}|^{\mu}dx+\lambda c_{2}\lvert\Omega\rvert.
\end{eqnarray*}

Choosing $ t=\delta_{2}'>\delta_{1} $, it follows that if $ \theta \in [\delta_{1},\delta_{2}'] $ and $ \delta_{2}'\gg 1 $, then 
\begin{eqnarray*}
\Phi_{u}^{1}(\delta_{2}',\theta)&\leq& (\delta_{2}')^{2}\|u^{-}\|_{\mathbb{H}_{0}^{s}(\Omega)}^{2}-2(\delta_{2}')^{2} \int _{\mathbb{R}^{2N}}\frac{u^{+}(x)u^{-}(y)}{|x-y|^{N+2s}}dx dy- (\delta_{2}')^{2_{s}^{*}}\int _{\Omega}|u^{+}|^{2_{s}^{*}}dx\\
&&- (\delta_{2}')^{\mu}\lambda c_{1}\int _{\Omega}|u^{+}|^{\mu}dx+\lambda c_{2}\lvert\Omega\rvert<0,
\end{eqnarray*} 
once $ \mu >2 $ and $ 2_{s}^{*}>2 $.

Analogously, it is possible to show that 
\begin{equation*}
\Phi_{u}^{2}(t,\theta)\leq t^{2}\|u^{-}\|_{\mathbb{H}_{0}^{s}(\Omega)}^{2}-2t\theta \int _{\mathbb{R}^{N}}\frac{u^{+}(x)u^{-}(y)}{|x-y|^{N+2s}}dx dy-t^{2_{s}^{*}}\int _{\Omega}|u^{-}|^{2_{s}^{*}}dx -\lambda c_{1} t^{\mu}\int _{\Omega}|u^{-}|^{\mu}dx +\lambda c_{2}\lvert\Omega\rvert.
\end{equation*}

Hence, choosing $ \delta_{2}>\delta_{2}'\gg 1 $, we get
\begin{equation}\label{Equa7}
\Phi_{u}^{2}(\delta_{2},\theta)<0\, \operatorname{and}\, \Phi_{u}^{2}(t,\delta_{2})<0,
\end{equation}
for all $ t,\theta\in [\delta_{1},\delta_{2}] $.

From \eqref{Equa6} and \eqref{Equa7}, the hypotheses of Miranda's Theorem are satisfied. Thus, applying Miranda's Theorem, there exist $ (t_{u},\theta_{u})\in R $ such that $ \Phi_{u}(t_{u},\theta_{u})=(0,0) $. It concludes the proof of this lemma. 

\end{proof}

\begin{lemma}\label{Lem4}
There exist $ \rho>0 $ such that 
\begin{equation*}
\|u^{+}\|_{\mathbb{H}_{0}^{s}(\Omega)}, \,\|u^{-}\|_{\mathbb{H}_{0}^{s}(\Omega)}\geq \rho,
\end{equation*}
for all $ u\in \mathcal{M}_{\lambda} $.
\end{lemma}
\begin{proof}
If $ u\in \mathcal{M}_{\lambda} $, then
\begin{equation*}
\int _{\mathbb{R}^{2N}}\frac{(u(x)-u(y))(u^{\pm}(x)-u^{\pm}(y))}{|x-y|^{N+2s}}dx dy=\int _{\Omega}|u^{\pm}|^{2_{s}^{*}}dx+\lambda \int _{\Omega}f(x,u^{\pm})u^{\pm}dx.
\end{equation*}

Since 
\begin{equation*}
\|u^{\pm}\|_{\mathbb{H}_{0}^{s}(\Omega)}^{2}\leq \int _{\mathbb{R}^{2N}}\frac{(u(x)-u(y))(u^{\pm}(x)-u^{\pm}(y))}{|x-y|^{N+2s}}dx dy,
\end{equation*}
using the Lemma \ref{Lem1} and the continuity of embedding $ \mathbb{H}_{0}^{s}(\Omega)\hookrightarrow L^{r}(\Omega) $, for all $ r\in [1,2_{s}^{*}] $, we deduce that 
\begin{eqnarray*}
\|u^{\pm}\|_{\mathbb{H}_{0}^{s}(\Omega)}^{2}&\leq&\int _{\Omega}|u^{\pm}|^{2_{s}^{*}}dx+\lambda \int _{\Omega}f(x,u^{\pm})u^{\pm}dx\\
&\leq&\int _{\Omega}|u^{\pm}|^{2_{s}^{*}}dx+\lambda \epsilon \int _{\Omega}|u^{\pm}|^{2}dx+\delta(\epsilon) \lambda \int _{\Omega}|u^{\pm}|^{q}dx\\
&\leq&\tilde{c}_{1}\|u^{\pm}\|_{\mathbb{H}_{0}^{s}(\Omega)}^{2_{s}^{*}}+\lambda \epsilon \tilde{c}_{1} \|u^{\pm}\|_{\mathbb{H}_{0}^{s}(\Omega)}^{2}+\delta(\epsilon) \lambda \tilde{c}_{1}\|u^{\pm}\|_{\mathbb{H}_{0}^{s}(\Omega)}^{q},
\end{eqnarray*}
for some $ \tilde{c}_{1}>0 $. Thus,
\begin{equation*}
(1-\epsilon \lambda \tilde{c}_{1})\|u^{\pm}\|_{\mathbb{H}_{0}^{s}(\Omega)}^{2}\leq \tilde{c}_{1}\|u^{\pm}\|_{\mathbb{H}_{0}^{s}(\Omega)}^{2_{s}^{*}}+\delta(\epsilon)\lambda \tilde{c}_{1}\|u^{\pm}\|_{\mathbb{H}_{0}^{s}(\Omega)}^{q}.
\end{equation*}

Taking $ \epsilon>0 $ such that $ 1-\epsilon \lambda \tilde{c}_{1}>0 $, we get the conclusion desired, since $ q>2 $ and $ 2_{s}^{*}>2 $.
\end{proof}

\begin{lemma}\label{Lem5}
There are $ c_{1} >0$ and $ c_{2}>0 $, with $ c_{1}=c_{1}(s,N,\mu) $ such that 
\begin{equation*}
c_{1}\|u\|_{\mathbb{H}_{0}^{s}(\Omega)}^{2}\leq I_{\lambda}(u)\leq c_{2}\|u\|_{\mathbb{H}_{0}^{s}(\Omega)}^{2},
\end{equation*}
for any $ u\in \mathcal{M}_{\lambda} $. 
\end{lemma}
\begin{proof}
Given $ u\in \mathcal{M}_{\lambda} $, since $ \mathcal{M}_{\lambda}\subset \mathcal{N}_{\lambda} $,  $ \langle I_{\lambda}'(u),u\rangle =0 $. Then, using \hyperlink{3}{(H3)}, 
\begin{eqnarray*}
I_{\lambda}(u)&=&\frac{1}{2}\|u\|_{\mathbb{H}_{0}^{s}(\Omega)}^{2}-\frac{1}{2_{s}^{*}}\int _{\Omega}|u|^{2_{s}^{*}}dx-\lambda \int _{\Omega}F(x,u)dx\\
&\geq&\frac{1}{2}\|u\|_{\mathbb{H}_{0}^{s}(\Omega)}^{2}-\frac{1}{2_{s}^{*}}\int _{\Omega}|u|^{2_{s}^{*}}dx-\frac{\lambda}{\mu}\int _{\Omega}f(x,u)u dx\\
&\geq& \frac{1}{2}\int _{\Omega} |u|^{2_{s}^{*}}dx + \frac{\lambda}{2}\int _{\Omega}f(x,u)u dx -\frac{1}{2_{s}^{*}}\int _{\Omega}|u|^{2_{s}^{*}}dx\\
&&-\frac{\lambda}{\mu}\int _{\Omega}f(x,u)u dx\\
&=&\left(\frac{1}{2}-\frac{1}{2_{s}^{*}} \right)\int _{\Omega}|u|^{2_{s}^{*}}dx + \lambda \left( \frac{1}{2}-\frac{1}{\mu} \right)\int _{\Omega}f(x,u)u dx\\
&\geq&\min \left\{\frac{1}{2}-\frac{1}{2_{s}^{*}},\frac{1}{2}-\frac{1}{\mu}\right\} \|u\|_{\mathbb{H}_{0}^{s}(\Omega)}^{2}.
\end{eqnarray*}
Thus, for any $ u\in \mathcal{M}_{\lambda} $,
\begin{equation*}
I_{\lambda}(u)\geq c_{1}\|u\|_{\mathbb{H}_{0}^{s}(\Omega)}^{2},
\end{equation*}
where $ c_{1}=\min \left\{\frac{1}{2}-\frac{1}{2_{s}^{*}},\frac{1}{2}-\frac{1}{\mu}\right\} $.

On the other hand, once by \hyperlink{3}{(H3)}, $ F(x,t)\geq 0$, for a.e. $ x\in \Omega $ and all $ t\in \mathbb{R} $, we get
\begin{eqnarray*}
I_{\lambda}(u)&=&\frac{1}{2}\|u\|_{\mathbb{H}_{0}^{s}(\Omega)}^{2}-\frac{1}{2_{s}^{*}}\int _{\Omega}|u|^{2_{s}^{*}}dx-\lambda \int _{\Omega}F(x,u)dx\\
&\leq&\frac{1}{2}\|u\|_{\mathbb{H}_{0}^{s}(\Omega)}^{2}=c_{2}\|u\|_{\mathbb{H}_{0}^{s}(\Omega)}^{2},
\end{eqnarray*}
for any $ u\in \mathcal{M}_{\lambda} $.
\end{proof}

\begin{proposition}\label{Prop1}
If $ \gamma_{\lambda}=\inf_{u\in \mathcal{M}_{\lambda}}I_{\lambda}(u) $, then
\begin{equation*}
\lim_{\lambda\rightarrow +\infty} \gamma _{\lambda}=0.
\end{equation*}
\end{proposition}

\begin{proof}
By Lemma \ref{Lem5}, it follows that $ I_{\lambda}(u)>0 $, for all $ u\in \mathcal{M}_{\lambda} $. Thus, $ I_{\lambda} $ is bounded below on $ \mathcal{M}_{\lambda} $ and, therefore, $ \gamma _{\lambda}=\inf _{u\in \mathcal{M}_{\lambda}}I_{\lambda}(u) $ is well-defined.

We fix $ u\in \mathbb{H}_{0}^{s}(\Omega) $, with $ u^{+}\neq 0 $ and $ u^{-}\neq 0 $. By Lemma \ref{lem3}, for each $ \lambda >0 $, there exist $ t_{\lambda} >0 $ and $ \theta_{\lambda}> 0$ such that 
\begin{equation*}
t_{\lambda}u^{+}+\theta _{\lambda}u^{-} \in \mathcal{M}_{\lambda}.
\end{equation*}
Also, by Lemma \ref{Lem5}, we have
\begin{eqnarray*}
0\leq \gamma_{\lambda}&=&\inf_{v\in \mathcal{M}_{\lambda}}I_{\lambda}(v)\leq I_{\lambda}(t_{\lambda}u^{+}+\theta _{\lambda}u^{-})\\
&\leq&\frac{1}{2}\|t_{\lambda}u^{+}+\theta _{\lambda}u^{-}\|_{\mathbb{H}_{0}^{s}(\Omega)}^{2}\\
&\leq&t_{\lambda}^{2}\|u^{+}\|_{\mathbb{H}_{0}^{s}(\Omega)}^{2}+ \theta _{\lambda}^{2}\|u^{-}\|_{\mathbb{H}_{0}^{s}(\Omega)}^{2}.
\end{eqnarray*}

Thus, it is enough to prove that $ t_{\lambda}\rightarrow 0 $ and $ \theta _{\lambda}\rightarrow 0 $, as $ \lambda \rightarrow +\infty $.
We consider the set
\begin{equation*}
Q_{u}=\{(t_{\lambda},\theta _{\lambda})\in [0,\infty)\times [0,\infty); \Phi_{u}(t_{\lambda},\theta _{\lambda})=(0,0), \lambda >0\},
\end{equation*}
where $ \Phi_{u} $ was defined in \eqref{Equa5}.
Since $ f(x,t)t\geq 0 $, for a.e. $ x\in \Omega $ and all $ t\in \mathbb{R} $ (by \hyperlink{3}{(H3)}) and $ t_{\lambda}u^{+}+\theta _{\lambda}u^{-}\in \mathcal{N}_{\lambda} $, we have
\begin{eqnarray*}
t_{\lambda}^{2_{s}^{*}}\int _{\Omega}|u^{+}|^{2_{s}^{*}}dx +\theta _{\lambda}^{2_{s}^{*}}\int _{\Omega} |u^{-}|^{2_{s}^{*}}dx&\leq& t_{\lambda}^{2_{s}^{*}}\int _{\Omega}|u^{+}|^{2_{s}^{*}}dx +\theta _{\lambda}^{2_{s}^{*}}\int _{\Omega} |u^{-}|^{2_{s}^{*}}dx\\
&&+\lambda \int _{\Omega}f(x,t_{\lambda} u^{+})t_{\lambda}u^{+}dx+\lambda \int _{\Omega}f(x,\theta _{\lambda}u^{-})\theta _{\lambda}u^{-}dx\\
&=&\|t_{\lambda}u^{+}+\theta _{\lambda}u^{-}\|_{\mathbb{H}_{0}^{s}(\Omega)}^{2}\\
&\leq&2t_{\lambda}^{2}\|u^{+}\|_{\mathbb{H}_{0}^{s}(\Omega)}^{2}+2\theta _{\lambda}^{2}\|u^{-}\|_{\mathbb{H}_{0}^{s}(\Omega)}^{2}.
\end{eqnarray*}
Then, since $ 2_{s}^{*}>2 $, it follows that $ Q_{u} $ is bounded.

Let $ \{\lambda _{j}\} \subset \mathbb{R}_{+}$ be such that $ \lambda _{j}\rightarrow +\infty $, as $ j\rightarrow +\infty $. Then, there exist $ t_{0} $ and $ \theta _{0} $ such that $ \{(t_{\lambda_{j}},\theta_{\lambda_{j}})\},$ up to a subsequence, still denoted by $ \{(t_{\lambda_{j}},\theta_{\lambda_{j}})\} ,$ conveges to $(t_{0},\theta _{0}),$ as $j\to +\infty.$

We will show that $ t_{0}=\theta _{0}=0 $. Suppose, by contradiction, that $ t_{0} >0$ or $ \theta _{0}> 0$. Since $ \{t_{\lambda_{j}}u^{+}+\theta _{\lambda _{j}}u^{-}\}\in \mathcal{N}_{\lambda} $, for any $ j\in \mathbb{N} $, we have
\begin{eqnarray}\label{Equa8}
\|t_{\lambda _{j}}u^{+}+\theta _{\lambda _{j}}u^{-}\|_{\mathbb{H}_{0}^{s}(\Omega)}^{2}&=&\lambda _{j}\int _{\Omega}f(x,t_{\lambda _{j}}u^{+}+\theta _{\lambda _{j}}u^{-})(t_{\lambda _{j}}u^{+}+\theta _{\lambda _{j}}u^{-})dx\nonumber\\
&& +\int _{\Omega}|t_{\lambda _{j}}u^{+}+\theta _{\lambda _{j}}u^{-}|^{2_{s}^{*}} dx .
\end{eqnarray}

Provided that $ t_{\lambda _{j}}u^{+}\rightarrow t_{0}u^{+} $ and $ \theta _{\lambda _{j}}u^{-}\rightarrow \theta _{0}u^{-} $ in $ \mathbb{H}_{0}^{s}(\Omega) $, by Lemma \ref{Lem1} and assumption \hyperlink{3}{(H3)}, we have
\begin{eqnarray*}
\int _{\Omega}f(x,t_{\lambda _{j}}u^{+}+\theta _{\lambda _{j}}u^{-})(t_{\lambda _{j}}u^{+}+\theta _{\lambda _{j}}u^{-})dx\rightarrow \int _{\Omega}f(x,t_{0}u^{+}+\theta _{0}u^{-})(t_{0}u^{+}+\theta _{0}u^{-})dx>0,
\end{eqnarray*}
as $ j\rightarrow +\infty $. Once $ \lambda_{j}\rightarrow +\infty $, as $ j\rightarrow +\infty $ and $\{ t_{\lambda _{j}}u^{+}+\theta _{\lambda _{j}}u^{-} \}$ is bounded in $ \mathbb{H}_{0}^{s}(\Omega) $, we have a contradiction with the equality \eqref{Equa8}. Thus, $ t_{0}=\theta_{0}=0 $.

Therefore $ \gamma _{\lambda _{j}}\rightarrow 0 $, as $ j\rightarrow +\infty $, which concludes the proof.
\end{proof}

Now, fixed $ u\in \mathbb{H}_{0}^{s}(\Omega) $, with $ u^{\pm}\neq 0$, we consider $ \varphi_{u}:[0,\infty)\times [0,\infty)\rightarrow \mathbb{R} $, defined by
\begin{equation*}
\varphi_{u}(t,\theta):=I_{\lambda}(tu^{+}+\theta u^{-})
\end{equation*}
and $ \psi_{u} :[0,\infty)\times [0,\infty)\rightarrow \mathbb{R}^{2} $, defined by
\begin{equation*}
\psi_{u}(t,\theta):=\left(\partial_{t} \varphi_{u}(t,\theta),\partial_{\theta}\varphi_{u}(t,\theta)\right)=\left(\langle I_{\lambda}'(tu^{+}+\theta u^{-}),u^{+}\rangle,\langle I_{\lambda}'(tu^{+}+\theta u^{-}),u^{-}\rangle\right).
\end{equation*}

Since $ f(x,\cdot) $ is of class $ C^{1} $, for a.e. $ x\in \Omega $, using \hyperlink{1}{(H1)}, it follows that $ \psi_{u} $ is also of class $ C^{1} $. 

\begin{lemma}\label{Lem6}
Let $ u\in \mathcal{M}_{\lambda} $. Then, 
\begin{itemize}
\item [(i)] $ \varphi_{u}(t,\theta)<\varphi_{u}(1,1) $, for all $ (t,\theta)\neq (1,1) $;
\item [(ii)] $ \det J_{(1,1)}\psi_{u} >0$, where $ J_{(1,1)}\psi_{u} $ is the Jacobian matrix of $ \psi_{u} $ in $ (1,1) $.
\end{itemize}
\end{lemma}

\begin{proof}
Firstly, we will proof the item (i). Since $ u\in \mathcal{M}_{\lambda} $, $ \langle I_{\lambda}'(u),u^{\pm}\rangle=0 $. Then, $ \psi_{u}(1,1)=(0,0) $ and thus $ (1,1) $ is a critical point of $ \varphi_{u} $ in $ [0,+\infty) \times [0,+\infty)$.

Note that, using Lemma \ref{Lem2}, we obtain
\begin{eqnarray}
\varphi _{u}(t,\theta)&=&\frac{1}{2}\|tu^{+}+\theta u^{-}\|_{\mathbb{H}_{0}^{s}(\Omega)}^{2}-\frac{1}{2_{s}^{*}}\int _{\Omega}|tu^{+}+\theta u^{-}|^{2_{s}^{*}}dx-\lambda \int _{\Omega}F(x,tu^{+}+\theta u^{-})dx \nonumber\\
&=&\frac{1}{2}t^{2}\|u^{+}\|_{\mathbb{H}_{0}^{s}(\Omega)}^{2}+\frac{1}{2}\theta ^{2}\|u^{-}\|_{\mathbb{H}_{0}^{s}(\Omega)}^{2}-2t\theta \int _{\mathbb{R}^{2N}}\frac{u^{+}(x)u^{-}(y)}{|x-y|^{N+2s}}dx dy\nonumber\\
&&-\frac{1}{2_{s}^{*}}t^{2_{s}^{*}}\int _{\Omega} |u^{+}|^{2_{s}^{*}}dx-\frac{1}{2_{s}^{*}} \theta^{2_{s}^{*}}\int _{\Omega}|u^{-}|^{2_{s}^{*}}dx\nonumber\\
&&-\lambda \int _{\Omega}F(x,tu^{+})dx -\lambda \int _{\Omega}F(x,\theta u^{-})dx\nonumber\\
&\leq&\frac{1}{2}t^{2}\|u^{+}\|_{\mathbb{H}_{0}^{s}(\Omega)}^{2}+\frac{1}{2}\theta ^{2}\|u^{-}\|_{\mathbb{H}_{0}^{s}(\Omega)}^{2}-2t \theta \int_{\mathbb{R}^{2N}}\frac{u^{+}(x)u^{-}(y)}{|x-y|^{N+2s}}dx dy\nonumber\\
&& -\frac{1}{2_{s}^{*}}\int _{\Omega}|u^{+}|^{2_{s}^{*}}dx-\frac{1}{2_{s}^{*}}\theta ^{2_{s}^{*}}\int _{\Omega}|u^{-}|^{2_{s}^{*}}dx\nonumber\\
&&-\lambda c_{1} t^{\mu} \int _{\Omega}|u^{+}|^{\mu}dx-\lambda c_{1}\theta ^{\mu}\int _{\Omega}|u^{-}|^{\mu}dx+ 2c_{2}\lvert\Omega\rvert.\label{Equa9}
\end{eqnarray}
Since $ \mu >2 $ and $ 2_{s}^{*}>2 $, it follows that
\begin{equation*}
\lim _{|(t,\theta)|\rightarrow +\infty}\varphi_{u}(t,\theta)= -\infty.
\end{equation*}

Provided that $ \varphi_{u} $ is continuous, we have that $ \varphi_{u} $ has a global maximum $ (t_{u},\theta_{u}) $ in $ [0,+ \infty)\times [0,+\infty) $. Our objective now is to show that $ (t_{u},\theta_{u})=(1,1) $ and this maximum is strict.

\noindent \textbf{Claim 1.} $ t_{u}>0 $ and $ \theta_{u}>0 $.

Suppose that $ \theta_{u}=0 $. Thus, $ t_{u}\neq 0 $ and $ \langle I_{\lambda}'(t_{u}u^{+}),t_{u}u^{+}\rangle=0 $. Then,

\begin{equation}\label{Equa10}
t_{u}^{2}\|u^{+}\|_{\mathbb{H}_{0}^{s}(\Omega)}^{2}=t_{u}^{2_{s}^{*}}\int _{\Omega}|u^{+}|^{2_{s}^{*}}dx+\lambda \int _{\Omega}f(x,t_{u}u^{+})t_{u}u^{+}dx.
\end{equation}

Since $ \langle I_{\lambda}'(u^{+}),u^{+}\rangle \leq \langle I_{\lambda}'(u),u^{+}\rangle=0$, we obtain
\begin{equation}\label{Equa11}
\|u^{+}\|_{\mathbb{H}_{0}^{s}(\Omega)}^{2}\leq \int _{\Omega}|u^{+}|^{2_{s}^{*}}dx + \lambda \int _{\Omega} f(x,u^{+})u^{+}dx .
\end{equation}

Therefore, by \eqref{Equa10} and \eqref{Equa11}, we have
\begin{equation}\label{Equa12}
\begin{array}{c}
{\displaystyle\left( 1-\frac{1}{t_{u}^{2}}\right) \|u^{+}\|_{\mathbb{H}_{0}^{s}(\Omega)}^{2}-\left (1-t_{u}^{2_{s}^{*}-2}\right )\int _{\Omega}|u^{+}|^{2_{s}^{*}}dx}
\\ \vspace{-.3cm} \\
\hspace{2cm}
\leq \lambda {\displaystyle\int _{\Omega}\left (\dfrac{f(x,u^{+})u^{+}}{(u^{+})^{2}}-\frac{f(x,t_{u}u^{+})t_{u}u^{+}}{(t_{u}u^{+})^{2}}\right )(u^{+})^{2}dx}.
\end{array}
\end{equation}
Assumption \hyperlink{4}{(H4)} and inequality \eqref{Equa12} imply that $ t_{u} \leq 1 $.

By Remark \ref{Obs1}, $ H(x,t)=f(x,t)t-2F(x,t)>0 $, for a.e. $ x\in \Omega $ and all $ t\neq 0,$ and also $ H(x,t) $ is increasing in $ t>0 $ for a.e. $ x\in \Omega $. In this case, we have
\begin{eqnarray*}
\varphi_{u}(t_{u},0)&=& I_{\lambda}(t_{u}u^{+})-\frac{1}{2}\langle I_{\lambda}'(t_{u}u^{+}), t_{u}u^{+}\rangle\\
&=&\left (\frac{1}{2}-\frac{1}{2_{s}^{*}}\right )\int _{\Omega}|t_{u}u^{+}|^{2_{s}^{*}}dx +\frac{\lambda}{2}\int _{\Omega}\left (f(x,t_{u}u^{+})t_{u}u^{+}-2F(x,t_{u}u^{+})\right )dx\\
&\leq& \left (\frac{1}{2}-\frac{1}{2_{s}^{*}}\right )\int _{\Omega}|u^{+}|^{2_{s}^{*}}dx +\frac{\lambda}{2}\int _{\Omega}\left(f(x,u^{+})u^{+}-2F(x,u^{+})\right)dx \\
&<&\left (\frac{1}{2}-\frac{1}{2_{s}^{*}}\right )\int _{\Omega}|u|^{2_{s}^{*}}dx+\frac{\lambda}{2}\int _{\Omega}\left (f(x,u^{+})u^{+}-2F(x,u^{+})\right )dx\\
&&\frac{\lambda}{2}\int _{\Omega}\left (f(x,u^{-})u^{-}-2F(x,u^{-})\right )dx\\
&=&\frac{1}{2}\|u\|_{\mathbb{H}_{0}^{s}(\Omega)}^{2}-\frac{1}{2_{s}^{*}}\int _{\Omega}|u|^{2_{s}^{*}}dx-\lambda \int _{\Omega}F(x,u)dx\\
&&-\frac{1}{2}\left (\|u\|_{\mathbb{H}_{0}^{s}(\Omega)}^{2}-\int _{\Omega}|u|^{2_{s}^{*}}dx-\lambda \int _{\Omega}f(x,u)u dx\right )\\
&=&I_{\lambda}(u)-\frac{1}{2}\langle I_{\lambda}'(u),u\rangle\\
&=&I_{\lambda}(u)=\varphi_{u}(1,1),
\end{eqnarray*}
which is a contradiction, because, in this case, $ (t_{u},0) $ is a point of maximum of $ \varphi_{u} $. Thus, $ \theta_{u}\neq 0 $. Analogously, it shows that $ t_{u}\neq 0 $.

\noindent \textbf{Claim 2.} $ t_{u} \leq 1$ and $ \theta_{u} \leq 1$.

Since $ (t_{u},\theta_{u})\in (0,+\infty)\times (0,+\infty)$, we have that $ (t_{u},\theta_{u}) $ is a critical point of $ \varphi_{u} $. Then, $ \psi_{u}(t_{u},\theta_{u})=(0,0) $ and, thus,
\begin{equation}\label{Equa13}
t_{u}^{2}\|u^{+}\|_{\mathbb{H}_{0}^{s}(\Omega)}^{2}-2t_{u}\theta_{u}\int_{\mathbb{R}^{2N}}\frac{u^{+}(x)u^{-}(y)}{|x-y|^{N+2s}}dx dy=t_{u}^{2_{s}^{*}}\int _{\Omega}|u^{+}|^{2_{s}^{*}}dx+\int _{\Omega}f(x,t_{u}u^{+})t_{u}u^{+}dx,
\end{equation}
\begin{equation*}
\theta_{u}^{2}\|u^{-}\|_{\mathbb{H}_{0}^{s}(\Omega)}^{2}-2t_{u}\theta_{u}\int_{\mathbb{R}^{2N}}\frac{u^{+}(x)u^{-}(y)}{|x-y|^{N+2s}}dx dy=\theta_{u}^{2_{s}^{*}}\int _{\Omega}|u^{-}|^{2_{s}^{*}}dx+\int _{\Omega}f(x,\theta_{u}u^{-})\theta_{u}u^{-}dx .
\end{equation*}

We suppose, without loss of generality, that $ t_{u}\geq \theta_{u} $. Then, using \eqref{Equa13},
\begin{equation}\label{Equa14}
\lambda \int _{\Omega}f(x,t_{u}u^{+})t_{u}u^{+}dx\leq t_{u}^{2}\|u^{+}\|_{\mathbb{H}_{0}^{s}(\Omega)}-2t_{u}^{2}\int _{\mathbb{R}^{2N}}\frac{u^{+}(x)u^{-}(y)}{|x-y|^{N+2s}}dx dy- \frac{1}{2_{s}^{*}}t_{u}^{2_{s}^{*}}\int _{\Omega}|u^{+}|^{2_{s}^{*}}dx.
\end{equation}

On the other hand, since that $ \langle I_{\lambda}'(u),u^{+}\rangle=0 $, 
\begin{equation}\label{Equa15}
\lambda \int _{\Omega}f(x,u^{+})u^{+}dx=\|u^{+}\|_{\mathbb{H}_{0}^{s}(\Omega)}^{2}-2\int _{\mathbb{R}^{2N}}\frac{u^{+}(x)u^{-}(y)}{|x-y|^{N+2s}}dx dy-\int _{\Omega}|u^{+}|^{2_{s}^{*}}dx.
\end{equation}

By \eqref{Equa14} and \eqref{Equa15}, we achieve
\begin{equation*}
\lambda \int _{\Omega}\left ( \frac{f(x,t_{u}u^{+})t_{u}u^{+}}{(t_{u}u^{+})^{2}}-\frac{f(x,u^{+})u^{+}}{(u^{+})^{2}}\right )(u^{+})^{2}dx\leq (1-t_{u}^{2_{s}^{*}-2})\int _{\Omega}|u^{+}|^{2_{s}^{*}}dx .
\end{equation*}
Again, by \hyperlink{4}{(H4)}, it follows that $ t_{u}\leq 1 $.

To conclude the proof of (i), it remains to show that $ \varphi_{u} $ does not have global maximum $ (t_{u}, \theta_{u}) $ in $ [0,1]\times [0,1]\setminus \{(1,1)\}  $. Indeed, supposing by contradiction that   $ (t_{u}, \theta_{u}) $ in $ [0,1]\times [0,1]\setminus \{(1,1)\}$ is a global maximum of $ \varphi_{u},$ we obtain
\begin{eqnarray*}
\varphi_{u}(t_{u},\theta_{u})&=&I_{\lambda}(t_{u}u^{+}+\theta_{u}u^{-})-\frac{1}{2}\langle I_{\lambda}'(t_{u}u^{+}+\theta_{u}u^{-}),t_{u}u^{+}+\theta_{u}u^{-}\rangle\\
&=&\left (\frac{1}{2}-\frac{1}{2_{s}^{*}}\right )\left (t_{u}^{2_{s}^{*}}\int _{\Omega}|u^{+}|^{2_{s}^{*}}dx+\theta_{u}^{2_{s}^{*}}\int _{\Omega}|u^{-}|^{2_{s}^{*}}dx\right )\\
&&+\frac{\lambda}{2}\int _{\Omega}f(x,t_{u}u^{+})t_{u}u^{+}-2F(x,t_{u}u^{+})dx\\
&&+\frac{\lambda}{2}\int _{\Omega}f(x,\theta_{u}u^{-})\theta_{u}u^{-}-2F(x,\theta_{u}u^{-})dx\\
&<&\left (\frac{1}{2}-\frac{1}{2_{s}^{*}} \right )\left (\int _{\Omega}|u^{+}|^{2_{s}^{*}}dx+\int_{\Omega}|u^{-}|^{2_{s}^{*}}dx\right)\\
&&+\frac{\lambda}{2}\int _{\Omega}f(x,u^{+})u^{+}-2F(x,u^{+})dx\\
&&+\frac{\lambda}{2}\int _{\Omega}f(x,u^{-})u^{-}-2F(x,u^{-})u^{-}dx\\
&=&\frac{1}{2}\|u\|_{\mathbb{H}_{0}^{s}(\Omega)}^{2}-\frac{1}{2_{s}^{*}}\int _{\Omega}|u^{+}|^{2_{s}^{*}}dx-\lambda \int _{\Omega}F(x,u)dx\\
&&-\frac{1}{2}\|u\|_{\mathbb{H}_{0}^{s}(\Omega)}^{2}-\frac{1}{2}\int _{\Omega}|u|^{2_{s}^{*}}dx-\frac{\lambda}{2}\int _{\Omega}f(x,u)u dx\\
&=&I_{\lambda}(u)-\frac{1}{2}\langle I_{\lambda}'(u),u\rangle=I_{\lambda}(u)=\varphi_{u}(1,1),
\end{eqnarray*}
which is a contradiction, because $ (t_{u},\theta_{u}) $ is a point of maximum of $ \varphi_{u} $. Hence, $ (1,1) $ is a strict maximum of $ \varphi_{u} $ and the proof of the item (i) is concluded.

Now, we proof the item (ii). Across some calculus, it is shows that 
\begin{equation*}
\partial _{t,t} ^{2}\varphi_{u}(t,\theta)=\|u^{+}\|_{\mathbb{H}_{0}^{s}(\Omega)}^{2}-(2_{s}^{*}-1)t^{2_{s}^{*}-2}\int _{\Omega}|u^{+}|^{2_{s}^{*}}dx-\lambda \int _{\Omega}\partial_{t}f(x,tu^{+})(u^{+})^{2}dx,
\end{equation*} 
\begin{equation*}
\partial _{\theta,\theta} ^{2}\varphi_{u}(t,\theta)=\|u^{-}\|_{\mathbb{H}_{0}^{s}(\Omega)}^{2}-(2_{s}^{*}-1)\theta^{2_{s}^{*}-2}\int _{\Omega}|u^{-}|^{2_{s}^{*}}dx-\lambda \int _{\Omega}\partial_{\theta}f(x,\theta u^{-})(u^{-})^{2}dx,
\end{equation*}
\begin{equation*}
\partial_{t,\theta}^{2}\varphi_{u}(t,\theta)=-2\int_ {\mathbb{R}^{2N}}\frac{u^{+}(x)u^{-}(y)}{|x-y|^{N+2s}}dx dy.
\end{equation*}

Now, note that
\begin{equation*}
J_{(1,1)}\psi_{u}=\left [
\begin{array}{llcc}
\partial_{t,t}^{2} \varphi_{u}(1,1) & \partial_{t,\theta}^{2} \varphi_{u}(1,1)\\
\partial_{\theta,t}^{2} \varphi_{u}(1,1) & \partial_{\theta,\theta}^{2} \varphi_{u}(1,1)
\end{array}
\right ].
\end{equation*}
Thus, using that $ u\in \mathcal{M}_{\lambda} $, we obtain
\begin{eqnarray*}
\det J_{(1,1)}\psi_{u}&=&\partial_{t,t}^{2} \varphi_{u}(1,1) \partial_{\theta,\theta}^{2} \varphi_{u}(1,1)- \partial_{t,\theta}^{2} \varphi_{u}(1,1)  \partial_{\theta,t}^{2} \varphi_{u}(1,1)\\
&=&\left (\|u^{+}\|_{\mathbb{H}_{0}^{s}(\Omega)}^{2}-(2_{s}^{*}-1)\int _{\Omega} |u^{+}|^{2_{s}^{*}}dx-\lambda \int _{\Omega}\partial _{t}f(x,u^{+})(u^{+})^{2}dx\right )\\
&& \times \left (\|u^{-}\|_{\mathbb{H}_{0}^{s}(\Omega)}^{2}-(2_{s}^{*}-1)\int _{\Omega}|u^{-}|^{2_{s}^{*}}dx-\lambda \int _{\Omega}\partial _{\theta}f((x,u^{-})(u^{-})dx\right )\\
&&-4\left (\int_{\mathbb{R}^{2N}}\frac{u^{+}(x)u^{-}(y)}{|x-y|^{N+2s}}dx dy\right )^{2}\\
&=&\left (2\int_{\mathbb{R}^{2N}}\frac{u^{+}(x)u^{-}(y)}{|x-y|^{N+2s}}dx dy-(2_{s}^{*}-2)\int _{\Omega} |u^{+}|^{2_{s}^{*}}dx - \lambda \int _{\Omega}\partial _{t}H(x,u^{+})u^{+}dx\right )\\
&&\times \left (2\int_{\mathbb{R}^{2N}}\frac{u^{-}(x)u^{-}(y)}{|x-y|^{N+2s}}dx dy-(2_{s}^{*}-2)\int _{\Omega} |u^{-}|^{2_{s}^{*}}dx - \lambda \int _{\Omega}\partial _{\theta} H(x,u^{-})u^{-} dx\right )\\
&&-4\left (\int _{\mathbb{R}^{2N}}\frac{u^{+}(x)u^{-}(y)}{|x-y|^{N+2s}}dx dy\right )^{2}>0,
\end{eqnarray*}
where $ H(x,t) $ was defined in Remark \ref{Obs1}. Thus, the proof of the item (ii) is concluded.
\end{proof}

\begin{lemma}\label{Lem7}
If $ \{u_{j}\} $ is bounded in $ \mathbb{H}_{0}^{s}(\Omega) $ and there exist $ u\in \mathbb{H}_{0}^{s}(\Omega) $ such that $ u_{j}(x)\rightarrow u(x)$, a.e. $ x\in \Omega $, then
\begin{equation*}
\lim_{j\rightarrow \infty}\|u_{j}\|_{\mathbb{H}_{0}^{s}(\Omega)}^{2}=\lim_{j\rightarrow \infty}\|u_{j}-u\|_{\mathbb{H}_{0}^{s}(\Omega)}^{2}+ \|u\|_{\mathbb{H}_{0}^{s}(\Omega)}^{2}.
\end{equation*} 
\end{lemma}
\begin{proof}
See \cite[Lemma 3.2]{I}.
\end{proof}

\section {Proof of Theorem \ref{Theo1}}\label{Section3}

By definition of $ \gamma _{\lambda} $, there exists a sequence $ \{u_{j}\}\subset \mathcal{M}_{\lambda} $ so that 
\begin{equation*}
\lim _{j\rightarrow \infty} I_{\lambda}(u_{j})=\gamma _{\lambda}.
\end{equation*}

From Lemma \ref{Lem5} and boundedness of $ \{I_{\lambda}(u_{j})\} $, there are $ c_{1}>0 $ and $ M>0 $ so that
\begin{equation*}
c_{1}\|u_{j}\|_{\mathbb{H}_{0}^{s}(\Omega)}^{2}\leq I_{\lambda}(u_{j})\leq M.
\end{equation*}
Thus, $ \{u_{j}\} $ is a bounded sequence in $ \mathbb{H}_{0}^{s}(\Omega) $. Then, up to a subsequence, still denoted by $ \{u_{j}\} $, there exist $ u\in \mathbb{H}_{0}^{s}(\Omega) $ so that
\begin{equation*}
u_{j}\rightharpoonup u.
\end{equation*}
Moreover, since the embedding $ \mathbb{H}_{0}^{s}(\Omega) \hookrightarrow L^{r}(\Omega) $ is compact, for all $ r\in [1,2_{s}^{*}) $, we have
\begin{equation*}
\begin{array}{lllc}
u_{j}\rightarrow u\, \operatorname{in}\, L^{r}(\Omega),\;
u_{j}(x)\rightarrow u(x)\, \operatorname{a.e.} \, x\in \Omega,\;
|u_{j}|\leq g,\, g \in L^{r}(\Omega).
\end{array}
\end{equation*}
Therefore, we deduce
\begin{equation*}
\begin{array}{llllc}
u_{j}^{\pm}\rightharpoonup u^{\pm}\, \operatorname{in} \, \mathbb{H}_{0}^{s}(\Omega),\;
u_{j}^{\pm}\rightarrow u^{\pm}\, \operatorname{in}\, L^{r}(\Omega),\;
u_{j}^{\pm}(x)\rightarrow u^{\pm}(x)\, \operatorname{a.e.} \, x\in \Omega,\;
|u_{j}^{\pm}|\leq g.
\end{array}
\end{equation*}

We consider the constant $ \beta=\frac{s}{N}S_{s}^{\frac{N}{2s}} $, where $ S_{s} $ was defined in \eqref{EquabestC}. By Proposition \ref{Prop1}, there exist $ \lambda _{*}>0 $ so that $ \gamma _{\lambda}<\beta $, for all $ \lambda \geq\lambda_{*} $.

We fix $ \lambda \geq \lambda _{*} $. By Lemma \ref{Lem6}, item (i), we obtain
\begin{equation*}
I_{\lambda}(t u_{j}^{+}+\theta u_{j}^{-})\leq I_{\lambda}(u_{j}),
\end{equation*}
for all $ t\geq 0 $ and all $ \theta \geq 0 $.

Note that
\begin{eqnarray*}
I_{\lambda}(t u_{j}^{+}+\theta u_{j}^{-})&=&\frac{t^{2}}{2}\|u_{j}^{+}\|_{\mathbb{H}_{0}^{}(\Omega)}^{2}+\frac{\theta ^{2}}{2}\|u_{j}^{-}\|_{\mathbb{H}_{0}^{}(\Omega)}^{2}-2t \theta \int _{\mathbb{R}^{2N}}\frac{u_{j}^{+}(x)u_{j}^{-}(y)}{|x-y|^{N+2s}}dx dy\\
&&-\frac{t^{2_{s}^{*}}}{2_{s}^{*}}\|u_{j}^{+}\|_{L^{2_{s}^{*}}(\Omega)}^{2_{s}^{*}}-\frac{\theta^{2_{s}^{*}}}{2_{s}^{*}}\|u_{j}^{-}\|_{L^{2_{s}^{*}}(\Omega)}^{2_{s}^{*}}\\
&&-\int _{\Omega}F(x,t u_{j}^{+})dx-\int _{\Omega}F(x,\theta u_{j}^{-})dx .
\end{eqnarray*}
By  using Brezis-Lieb Lemma \cite[Theorem 1]{J}, Lemma \ref{Lem7},  Fatou's Lemma, and Lemma \ref{Lem1}, we obtain
\begin{eqnarray*}
\liminf _{j\rightarrow \infty} I_{\lambda}(tu_{j}^{+}+\theta u_{j}^{-})&=&\frac{t^{2}}{2}\lim _{j\rightarrow \infty}\left (\|u_{j}^{+}-u^{+}\|_{\mathbb{H}_{0}^{s}(\Omega)}^{2}+\|u^{+}\|_{\mathbb{H}_{0}^{s}(\Omega)}^{2}\right )\\
&& +\frac{\theta^{2}}{2}\lim _{j\rightarrow \infty}\left (\|u_{j}^{-}-u^{-}\|_{\mathbb{H}_{0}^{s}(\Omega)}^{2}+\|u^{-}\|_{\mathbb{H}_{0}^{s}(\Omega)}^{2}\right )\\
&&+\liminf _{j\rightarrow \infty} \left (-2t\theta \int _{\mathbb{R}^{2N}}\frac{u_{j}^{+}(x)u_{j}^{-}(y)}{|x-y|^{N+2s}}dx dy\right )\\
&&-\frac{t^{2_{s}^{*}}}{2_{s}^{*}} \lim _{j\rightarrow \infty}\left (\|u_{j}^{+}-u^{+}\|_{L^{2_{s}^{*}}(\Omega)}^{2_{s}^{*}}+\|u^{+}\|_{L^{2_{s}^{*}}(\Omega)}^{2_{s}^{*}}\right )\\
&&-\frac{\theta^{2_{s}^{*}}}{2_{s}^{*}} \lim _{j\rightarrow \infty}\left (\|u_{j}^{-}-u^{-}\|_{L^{2_{s}^{*}}(\Omega)}^{2_{s}^{*}}+\|u^{-}\|_{L^{2_{s}^{*}}(\Omega)}^{2_{s}^{*}}\right )\\
&&-\int _{\Omega}F(x,t u^{+})dx-\int _{\Omega}F(x,\theta u^{-})dx\\
&\geq&\frac{t^{2}}{2}\lim _{j\rightarrow \infty}\|u_{j}^{+}-u^{+}\|_{\mathbb{H}_{0}^{s}(\Omega)}^{2}-\frac{t^{2_{s}^{*}}}{2_{s}^{*}}\lim _{j\rightarrow \infty}\|u_{j}^{+}-u^{+}\|_{L^{2_{s}^{*}}(\Omega)}^{2_{s}^{*}}\\
&&+\frac{\theta ^{2}}{2}\lim _{j\rightarrow \infty}\|u_{j}^{-}-u^{-}\|_{\mathbb{H}_{0}^{s}(\Omega)}^{2}-\frac{\theta ^{2_{s}^{*}}}{2_{s}^{*}}\lim _{j\rightarrow\infty}\|u_{j}^{-}-u^{-}\|_{L^{2_{s}^{*}}(\Omega)}^{2_{s}^{*}}\\
&&+\frac{t^{2}}{2}\|u^{+}\|_{\mathbb{H}_{0}^{s}(\Omega)}^{2}+\frac{\theta^{2}}{2}\|u^{-}\|_{\mathbb{H}_{0}^{s}(\Omega)}^{2} -2t\theta \int _{\mathbb{R}^{2N}}\frac{u^{+}(x)u^{-}(y)}{|x-y|^{N+2s}}dx dy\\
&&-\frac{t^{2_{s}^{*}}}{2_{s}^{*}}\|u^{+}\|_{L^{2_{s}^{*}}(\Omega)}^{2_{s}^{*}}-\frac{\theta ^{2_{s}^{*}}}{2_{s}^{*}}\|u^{-}\|_{L^{2_{s}^{*}}(\Omega)}^{2_{s}^{*}}-\int _{\Omega}F(x,t u^{+}+\theta u^{-})dx\\
&=&A^{+}\frac{t^{2}}{2}-B^{+}\frac{t ^{2_{s}^{*}}}{2_{s}^{*}}+A^{-}\frac{\theta ^{2}}{2}-B^{-}\frac{\theta ^{2_{s}^{*}}}{2_{s}^{*}}+I_{\lambda}(t u^{+}+\theta u^{-}),
\end{eqnarray*}
where 
\begin{equation*}
\begin{array}{llc}
A^{\pm}=\lim \limits_{j\rightarrow \infty} \|u_{j}^{\pm}-u^{\pm}\|_{\mathbb{H}_{0}^{s}(\Omega)}^{2},\\
B^{\pm}=\lim \limits_{j \rightarrow \infty} \|u_{j}^{\pm}-u^{\pm}\|_{L^{2_{s}^{*}}(\Omega)}.
\end{array}
\end{equation*}
Therefore, we achieve
\begin{equation}\label{Equa16}
I_{\lambda}(t u^{+}+\theta u^{-})+A^{+}\frac{t^{2}}{2}-B^{+}\frac{t ^{2_{s}^{*}}}{2_{s}^{*}}+A^{-}\frac{\theta ^{2}}{2}-B^{-}\frac{\theta ^{2_{s}^{*}}}{2_{s}^{*}}\leq \gamma _{\lambda},
\end{equation}
for all $ t\geq 0 $ and all $ \theta \geq 0 $.

\noindent \textbf{Claim 1.} $ u^{+}\neq 0 $ and $ u^{-}\neq 0 $.

We will only show that $ u^{+}\neq 0 $, because to show that $ u^{-}\neq 0$ it is analogous. Suppose, by contradiction, that $ u^{+}=0 $. Then, we need to analyse two cases:

\noindent \emph{Case 1:} $ B^{+}=0 $.

In this case, if $ A^{+}=0 $, by Lemma \ref{Lem4}, we obtain $ \|u^{+}\|_{\mathbb{H}_{0}^{s}(\Omega)}>0 $, which contradicts our supposition. If $ A^{+}>0 $, by \eqref{Equa16}, $ A^{+}\frac{t^{2}}{2}\leq \gamma _{\lambda} $, for all $ t\geq 0 $, which is false. Anyway, we have a contradiction. 

\noindent\emph{Case 2:} $ B^{+}>0 $.

In this case, by definition of $ S_{s},$ it follows that
\begin{equation*}
\beta =\frac{s}{N}S_{s}^{\frac{N}{2s}}\leq \frac{s}{N}\left (\frac{A^{+}}{(B^{+})^{\frac{2}{2_{s}^{*}}}}\right )^{\frac{N}{2s}}.
\end{equation*}
On the other hand, 
\begin{equation*}
\frac{s}{N}\left (\frac{A^{+}}{(B^{+})^{\frac{2}{2_{s}^{*}}}}\right )^{\frac{N}{2s}}=\max_{t\geq 0} \left (A^{+}\frac{t^{2}}{2}-B^{+}\frac{t^{2_{s}^{*}}}{2_{s}^{*}}\right ).
\end{equation*}
Thus, recalling that $ \gamma_{\lambda}<\beta $ and using \eqref{Equa16}, it follows that
\begin{equation*}
\beta \leq \max_{t\geq 0} \left (A^{+}\frac{t^{2}}{2}-B^{+}\frac{t^{2_{s}^{*}}}{2_{s}^{*}}\right )<\beta,
\end{equation*}
which is a contradiction. Therefore, we conclude that $ u^{+}\neq 0 $.

\noindent \textbf{Claim 2.} $ B^{+}=0 $ and $ B^{-}=0 $.

As in claim 1, we will only show that $ B^{+}=0 $, because to show that $ B^{-}=0 $ it is similar. We suppose, by contradiction, that $ B^{+}>0 $ and we consider the following cases:

\noindent \emph{Case 1:} $ B^{-}>0 $.

Let $ t_{\max} $ and $ \theta_{\max} $ be so that
\begin{equation*}
\max_{t\geq 0}\left (A^{+}\frac{t^{2}}{2}-B^{+}\frac{t^{2_{s}^{*}}}{2_{s}^{*}}\right )=A^{+}\frac{t_{\max}^{2}}{2}-B^{+}\frac{t_{\max}^{2_{s}^{*}}}{2_{s}^{*}},
\end{equation*}

\begin{equation*}
\max_{\theta\geq 0}\left (A^{-}\frac{\theta^{2}}{2}-B^{-}\frac{\theta^{2_{s}^{*}}}{2_{s}^{*}}\right )=A^{-}\frac{\theta_{\max}^{2}}{2}-B^{-}\frac{\theta_{\max}^{2_{s}^{*}}}{2_{s}^{*}}.
\end{equation*}

Since $  [0,t_{\max} ]\times  [0,\theta _{\max} ] $ is compact and $ \varphi _{u} $ is continuous, there exist $  (t_{u},\theta_{u} )\in  [0,t_{\max} ]\times  [0,\theta _{\max} ] $ so that 

\begin{equation*}
\varphi _{u}(t_{u},\theta_{u})=\max _{(t,\theta)\in [0,t_{\max}]\times [0,\theta _{\max}] } \varphi _{u}(t,\theta ).
\end{equation*}

We will show that $ (t_{u},\theta_{u} )\in  (0,t_{\max} )\times  (0,\theta _{\max} ) $, namely, we will show that $ (t_{u},\theta_{u} )\notin \partial ([0,t_{\max} ]\times  [0,\theta _{\max} ]) $.

Note that, if $ \theta \ll 1 $, then
\begin{eqnarray*}
\varphi_{u}(t,0)=I_{\lambda}(tu^{+})&<&I_{\lambda}(t u^{+})+I_{\lambda}(\theta u^{-})\\
&\leq&I_{\lambda}(tu^{+}+\theta u^{-})=\varphi _{u}(t,\theta),
\end{eqnarray*}
for all $ t\in [0,t_{\max}]$. Hence, there exist $ \theta _{0}\in [0,\theta _{\max}] $ so that
\begin{equation*}
\varphi_{u}(t,0)<\varphi_{u}(t,\theta _{0}),
\end{equation*}
for all $ t\in [0,t_{\max}] $. 

Thus, any point of the form $ (t,0) $, with $ 0\leq t\leq t_{\max} $, is not the maximizer of $ \varphi_{u} $. Hence, $ (t_{u},\theta_{u})\notin [0,t_{\max}]\times \{0\} $. Analogously, it is shows that $ (t_{u},\theta_{u})\notin \{0\}\times [0,\theta_{\max}] $.

Now, across a simple analysis, we obtain
\begin{equation}\label{Equa17}
A^{+}\frac{t^{2}}{2}-B^{+}\frac{t^{2_{s}^{*}}}{2_{s}^{*}}>0 ,
\end{equation}
\begin{equation}\label{Equa18}
A^{-}\frac{\theta^{2}}{2}-B^{-}\frac{\theta^{2_{s}^{*}}}{2_{s}^{*}}> 0,
\end{equation}
for all $ t\in (0,t_{\max}] $ and all $ \theta\in (0,\theta_{\max}] $. Then,
\begin{equation*}
\beta \leq A^{+}\frac{t_{\max}^{2}}{2}-B^{+}\frac{t_{\max}^{2_{s}^{*}}}{2_{s}^{*}}+ A^{-}\frac{\theta^{2}}{2}-B^{-}\frac{\theta^{2_{s}^{*}}}{2_{s}^{*}},
\end{equation*}
\begin{equation*}
\beta\leq A^{+}\frac{t^{2}}{2}-B^{+}\frac{t^{2_{s}^{*}}}{2_{s}^{*}}+ A^{-}\frac{\theta_{\max}^{2}}{2}-B^{-}\frac{\theta_{\max}^{2_{s}^{*}}}{2_{s}^{*}},
\end{equation*}
for all $ t\in [0,t_{\max}] $ and all $ \theta\in [0,\theta_{\max}] $.

Thus, by \eqref{Equa16}, it follows that
\begin{equation*}
\varphi_{u}(t,\theta _{\max})\leq 0,
\end{equation*}
\begin{equation*}
\varphi_{u}(t_{\max},\theta)\leq 0,
\end{equation*}
for all $ t\in [0,t_{\max}] $ and all $ \theta \in [0,\theta_{\max}] $. So, $ (t_{u},\theta_{u})\notin \{t_{\max}\}\times [0,\theta _{\max}] $ and $ (t_{u},\theta _{u})\notin [0,t_{\max}]\times \{\theta _{\max}\} $. Thus, we conclude that $ (t_{u},\theta_{u})\notin \partial ([0,t_{\max}]\times [0,\theta _{\max}]) $.

Being $ (t_{u},\theta_{u}) $ an inner maximizer of $ \varphi_{u} $, it follows that $ (t_{u},\theta_{u}) $ is a critical point of $ \varphi_{u} $. Thus, $ \psi_{u}(t_{u},\theta_{u})=(0,0) $, namely, $ t_{u}u^{+}+\theta_{u}u^{-}\in \mathcal{M}_{\lambda} $, with $ t_{u}\in (0,t_{\max}) $ and $ \theta_{u}\in (0,\theta _{\max}) $.

Therefore, using \eqref{Equa17} and \eqref{Equa18}, we obtain
\begin{eqnarray*}
\gamma_{\lambda}&\geq& A^{+}\frac{(t_{u})^{2}}{2}-B^{+}\frac{(t_{u})^{2_{s}^{*}}}{2_{s}^{*}}+ A^{-}\frac{(\theta_{u})^{2}}{2}-B^{-}\frac{(\theta_{u})^{2_{s}^{*}}}{2_{s}^{*}}+I_{\lambda}(t_{u}u^{+}+\theta_{u}u^{-})\\
&>&I_{\lambda}(t_{u}u^{+}+\theta_{u}u^{-})\geq \gamma _{\lambda},
\end{eqnarray*}
which is a contradiction.

\noindent \emph{Case 2:} $ B^{-}=0 $.

In this case, we can maximize in $ [0,t_{\max}]\times [0,+\infty) $. Indeed, it is possible to show that there exist $ \theta _{0}\in [0,+\infty) $ so that $ I_{\lambda}(tu^{+}+\theta u^{-})\leq 0 $, for all $ (t,\theta)\in [0,t_{\max}]\times [\theta _{0},+\infty) $. Hence, there exist $ (t_{u},\theta_{u})\in [0,t_{\max}]\times [0,+\infty) $ so that
\begin{equation*}
\varphi_{u}(t_{u},\theta_{u})=\max_{(t,\theta)\in [0,t_{\max}]\times [0,+\infty)}\varphi_{u}(t,\theta).
\end{equation*}

We need to ascertain that $ (t_{u},\theta_{u})\in (0,t_{\max})\times (0,+\infty) $.

Since, for $ \theta \ll 1 $, 
\begin{equation*}
\varphi_{u}(t,0)<\varphi_{u}(t,\theta), 
\end{equation*}
for all $ t\in [0,t_{\max}] $, it follows that $ (t_{u},\theta_{u})\notin [0,t_{\max}]\times \{0\} $.
Also, for $ t\ll 1 $, 
\begin{equation*}
\varphi_{u}(0,\theta)<\varphi_{u}(t,\theta),
\end{equation*}
for all $ \theta \in [0,+\infty) $. Then, $ (t_{u},\theta_{u})\notin \{0\}\times [0,+\infty) $.
Moreover, note that
\begin{equation*}
\beta\leq A^{+}\frac{t_{\max}^{2}}{2}-B^{+}\frac{t_{\max}^{2_{s}^{*}}}{2_{s}^{*}}+A^{-}\frac{\theta^{2}}{2},
\end{equation*}
for all $ \theta \in [0,+\infty) $. Then, 
\begin{equation*}
\varphi_{u}(t_{\max},\theta)\leq 0,
\end{equation*}
for all $ \theta \in [0,+\infty) $. Thus, $ (t_{u},\theta_{u})\notin \{t_{\max}\}\times [0,+\infty) $.
Hence, $ (t_{u},\theta_{u})\in (0,t_{\max})\times (0,+\infty) $. Consequently, $ (t_{u},\theta_{u}) $ is an inner maximizer of $ \varphi_{u} $ in $ [0,t_{\max}]\times [0,+\infty) $ and thus, $ \psi_{u}(t_{u},\theta_{u})=(0,0) $, namely, $ t_{u}u^{+}+\theta_{u}u^{-}\in \mathcal{M}_{\lambda} $.

Therefore, using \eqref{Equa17}, we obtain
\begin{eqnarray*}
\gamma_{\lambda}&\geq& A^{+}\frac{(t_{u})^{2}}{2}-B^{+}\frac{(t_{u})^{2_{s}^{*}}}{2_{s}^{*}}+A^{-}\frac{(\theta_{u})^{2}}{2}+I_{\lambda}(t_{u}u^{+}+\theta_{u}u^{-})\\
&>&I_{\lambda}(t_{u}u^{+}+\theta_{u}u^{-})\geq \gamma _{\lambda},
\end{eqnarray*}
which is a contradiction. Then, $ B^{+}=0,$ which concludes the proof of Claim 2.

\noindent \textbf{Claim 3.} The infimum $ \gamma _{\lambda} $ is achieved.

Since $ u^{+}\neq 0 $ and $ u^{-}\neq 0 $, by Lemma \ref{lem3}, there are $ t_{u}>0 $ and $ \theta_{u}>0 $ so that $ t_{u}u^{+}+\theta_{u}u^{-} \in \mathcal{M}_{\lambda}$. Then, as $ u_{j}\in \mathcal{M}_{\lambda} $, for any $ j\in \mathbb{N} $, by Lemma \ref{Lem6}, item (i),
\begin{equation*}
I_{\lambda}(t_{u}u_{j}^{+}+\theta_{u}u_{j}^{-})\leq I_{\lambda}(u_{j}^{+}+u_{j}^{-})=I_{\lambda}(u_{j}).
\end{equation*}

Since $ B^{\pm}=0 $ and the norm in $ \mathbb{H}_{0}^{s}(\Omega) $ is lower semicontinuous, it follows that
\begin{eqnarray*}
\gamma _{\lambda}&\leq &I_{\lambda}(t_{u}u^{+}+\theta_{u}u^{-})\\
&\leq&\liminf_{j\rightarrow \infty} I_{\lambda}(t_{u}u_{j}^{+}+\theta_{u}u_{j}^{-})\\
&\leq& \liminf _{j\rightarrow \infty}I_{\lambda}(u_{j})=\gamma _{\lambda}.
\end{eqnarray*}
Hence, the infimum $ \gamma _{\lambda} $ is achieved by $ t_{u}u^{+}+\theta_{u}u^{-} \in \mathcal{M}_{\lambda} $, which concludes the proof of Claim 3. 

Using the same arguments presented in \cite[Theorem 1.3]{E}, where Lemma \ref{Lem6} is required, it follows that $ v=t_{u}u^{+}+\theta_{u}u^{-} $ is a critical point of the functional $ I_{\lambda} $. Therefore, $ v\in \mathcal{M}_{\lambda} $ is a sign-changing solution for the equation \eqref{Equa1}.

\bibliographystyle{amsplain}

\end{document}